\documentclass{amsart}

\newtheorem{theorem}{Theorem}[section]
\newtheorem{lemma}[theorem]{Lemma}
\newtheorem{proposition}[theorem]{Proposition}

\theoremstyle{definition}

\newtheorem{example}[theorem]{Example}
\newtheorem{remark}[theorem]{Remark}

\usepackage{amscd,amssymb}

\begin{document}

\title[Noncommutative invariant theory]
{Noncommutative invariant theory\\
of symplectic and orthogonal groups}

\author[Vesselin Drensky and Elitza Hristova]
{Vesselin Drensky and Elitza Hristova}
\address{Institute of Mathematics and Informatics,
Bulgarian Academy of Sciences,
Acad. G. Bonchev Str., Block 8,
1113 Sofia, Bulgaria}
\email{drensky@math.bas.bg, e.hristova@math.bas.bg}

\thanks{Partially supported by Project DFNP 17-90/28.07.2017
of the Young Researchers Program of the Bulgarian Academy of
Sciences.}



\begin{abstract} We present a method for computing the Hilbert series of the algebra of invariants of
the complex symplectic and orthogonal groups acting on graded noncommutative algebras with homogeneous components
which are polynomial modules of the general linear group. We apply our method to compute the Hilbert series
for different actions of the symplectic and orthogonal groups on
the relatively free algebras of the varieties of associative algebras generated, respectively,
by the Grassmann algebra and the algebra of $2\times 2$ upper triangular matrices.
These two varieties are remarkable with the property that they are the only minimal varieties of exponent 2.
\end{abstract}

\maketitle

\section{Introduction}
Many of the results of the present paper hold oven an arbitrary field of characteristic 0.
But in the spirit of classical invariant theory we shall work over the field $\mathbb C$ of complex numbers.
A possible way to develop noncommutative invariant theory is the following. One considers
the tensor algebra
\[
T(W_m)=\sum_{n\geq 0}W_m^{\otimes_n}={\mathbb C}\oplus W_m\oplus (W_m\otimes W_m)\oplus (W_m\otimes W_m\otimes W_m)\oplus\cdots
\]
of the $m$-dimensional complex vector space $W_m={\mathbb C}X_m$, $m\geq 2$, with basis $X_m=\{x_1,\ldots,x_m\}$
with the canonical action of the general linear group $\text{GL}_m({\mathbb C})$. Then the action of $\text{GL}_m({\mathbb C})$ is extended diagonally
on $T(W_m)$. For a subgroup $G$ of $\text{GL}_m({\mathbb C})$ one studies the $G$-invariants of the factor algebra
$T(W_m)/I$, where $I$ is an ideal of $T(W_m)$ which is stable under the action of $\text{GL}_m({\mathbb C})$.
Maybe the most interesting algebras to study are the relatively free algebras $F_m({\mathfrak R})$
of varieties of unitary associative algebras $\mathfrak R$.
One considers the free unitary associative algebra ${\mathbb C}\langle X_m\rangle$ which is isomorphic
to the tensor algebra $T(W_m)$ and the ideal $I$ consists of all polynomial identities of the variety $\mathfrak R$.
Relatively free associative algebras share a lot of nice properties typical for polynomial algebras.
More generally, one may consider the free nonassociative (unitary or nonunitary) algebra ${\mathbb C}\{X_m\}$ modulo the ideal of the polynomial identities
of a variety $\mathfrak R$ of not necessarily associative algebras. For a subgroup $G$ of $\text{GL}_m({\mathbb C})$ one studies the algebra of $G$-invariants
\[
(T(W_m)/I)^G=\{f(W_m)\in T(W_m)/I\mid g(f)=f\text{ for all }g\in G\}.
\]
The algebra $(T(W_m)/I)^G$ is graded and its Hilbert (or Poincar\'e) series is
\[
H((T(W_m)/I)^G,z)=\sum_{n\geq 0}\dim (W_m^{\otimes_n}/(W_m^{\otimes_n}\cap I))^Gz^n.
\]
As in the case of classical invariant theory the computation of the Hilbert series of the algebra of $G$-invariants is one of the main problems
in noncommutative invariant theory.

In our paper we consider a more general situation. We have a direct sum of polynomial $\text{GL}_d({\mathbb C})$-modules
\[
W=\sum_{n\geq 0}W^{(n)}=W^{(0)}\oplus W^{(1)}\oplus W^{(2)}\oplus\cdots.
\]
Then $W$ has a canonical ${\mathbb N}_0$-grading with $W^{(n)}$ being the homogeneous component of degree $n$.
The $\text{GL}_d({\mathbb C})$-action induces an ${\mathbb N}_0^d$-grading. The homogeneous component $W^{(n,\alpha)}$ of $W^{(n)}$ of degree $\alpha=(\alpha_1,\ldots,\alpha_d)$
consists of all elements $w$ of $W^{(n)}$ with the property
\[
\text{diag}(\xi_1,\ldots,\xi_d)(w)=\xi_1^{\alpha_1}\cdots \xi_d^{\alpha_d}w,\text{ where }
\text{diag}(\xi_1,\ldots,\xi_d)=\left(\begin{matrix}
\xi_1&0&\cdots&0\\
0&\xi_2&\cdots&0\\
\vdots&\vdots&\ddots&\vdots\\
0&0&\cdots&\xi_d\\
\end{matrix}\right).
\]
(When we consider factor algebras of the free associative algebra ${\mathbb C}\langle X_d\rangle$ or, equivalently, of the tensor algebra of $W_d$
with the canonical action of $\text{GL}_d({\mathbb C})$,
the ${\mathbb N}_0^d$-grading of ${\mathbb C}\langle X_d\rangle$ is the usual one
which counts the number of entries of $x_i$ in the monomials of ${\mathbb C}\langle X_d\rangle$.)
Then the Hilbert series of $W$ is
\[
H_{\text{GL}_d}(W,T_d,z)=H_{\text{GL}_d}(W,t_1,\ldots,t_d,z)=\sum_{\alpha_i,n\geq 0}\dim W^{(n,\alpha)}t_1^{\alpha_1}\cdots t_d^{\alpha_d}z^n.
\]
It is easy to see that for any subgroup $G$ of $\text{GL}_d({\mathbb C})$ the Hilbert series $H_{\text{GL}_d}(W,T_d,z)$ determines the Hilbert series
\[
H(W^G,z)=\sum_{n\geq 0}\dim (W^{(n)})^Gz^n
\]
of the vector space of $G$-invariants
\[
W^G=\sum_{n\geq 0}(W^{(n)})^G.
\]
Recently Domokos and one of the authors of the present paper \cite{DD2} have shown that
if the Hilbert series $H_{\text{GL}_d}(W,T_d,z)$ is a rational function of a special kind (the so called nice rational function),
then the Hilbert series $H(W^G,z)$ is a rational function for a large class of groups $G$, including the cases
when $G$ is reductive or $G$ is a maximal unipotent subgroup of a reductive subgroup of $\text{GL}_m({\mathbb C})$.
In particular this holds when $W$ is a relatively free associative algebra $F_m({\mathfrak R})$ and $\mathfrak R$ is a proper subvariety
of the variety of associative algebras.

In our paper we consider an arbitrary $\displaystyle W=\sum_{n\geq 0}W^{(n)}$ which is a sum of polynomial $\text{GL}_d({\mathbb C})$-modules
and assume that we know the decomposition of all $W^{(n)}$ into a sum of irreducible components.
We present a method which allows to find the Hilbert series $H(W^G,z)$
when $G$ is one of the classical subgroups of $\text{GL}_d({\mathbb C})$ -- the symplectic group $\text{Sp}_d({\mathbb C})$ (when $d$ is even),
the orthogonal group $\text{O}_d({\mathbb C})$, and the special orthogonal group $\text{SO}_d({\mathbb C})$.
The approach is similar to the case when $G$ is the special linear group $\text{SL}_d({\mathbb C})$
or the unitriangular group $\text{UT}_d({\mathbb C})$ considered in \cite{BBDGK}.
Unfortunately we know the $\text{GL}_d({\mathbb C})$-module structure of $W$ in very few cases.
Examples of such $W$ are the relatively free algebras $F_d({\mathfrak R})$ where $\mathfrak R$ is the variety $\mathfrak G$
generated by the Grassmann (or exterior) algebra $E=\Lambda(W_{\infty})$ on the infinitely dimensional vector space $W_{\infty}$,
the variety ${\mathfrak T}$ generated by the algebra of $2\times 2$ upper triangular matrices,
the variety $\mathfrak M$ generated by the algebra of $2\times 2$ matrices,
the variety $\text{var}(E\otimes E)$ generated by the tensor square of $E$,
as well as the relatively free algebras of some varieties of Lie and Jordan algebras.
Other examples are the Grassmann algebra $E_d=\Lambda(W_d)$,
the universal enveloping algebra $U(F_d({\mathfrak N}_2))$ of the free nilpotent of class 2 Lie algebra $F_d({\mathfrak N}_2)$
and the $d$-generated generic Clifford algebra $\text{Cl}_d(W_p)$ of the $p$-dimensional vector space $W_p$, $p\in{\mathbb N}\cup{\infty}$.
If $W=T(W_m)/I$ is a factor algebra of the tensor algebra $T(W_m)$ and $W_m$ is a $\text{GL}_d({\mathbb C})$-module with $m\not=d$,
then the $\text{GL}_d({\mathbb C})$-module structure of $W$ may be quite complicated
and it is not a trivial task to find it. In this case if we know the $\text{GL}_d({\mathbb C})$-module structure of $W_m$
and the Hilbert series $H_{\text{GL}_m}(T(W_m)/I,T_m,z)$ with respect to the canonical ${\mathbb N}_0^m\times {\mathbb N}_0$-grading, then we can compute
the Hilbert series $H_{\text{GL}_d}(T(W_m)/I,T_d,z)$ with respect to the ${\mathbb N}_0^d\times {\mathbb N}_0$-grading
induced by the action of $\text{GL}_d({\mathbb C})$.
In the special case when the Hilbert series $H_{\text{GL}_m}(T(W_m)/I,T_m,z)$ is a nice rational function the methods described in \cite{BBDGK} allow
to find the decomposition of $T(W_m)/I$ as a $\text{GL}_d({\mathbb C})$-module and hence to compute the Hilbert series $H((T(W_m)/I)^G,z)$
for $G=\text{Sp}_d({\mathbb C}), \text{O}_d({\mathbb C})$, and $\text{SO}_d({\mathbb C})$.

As an illustration of our approach we apply our results to the algebra of invariants when the classical group acts
on the relatively free algebras $F_d({\mathfrak G})$ and $F_d({\mathfrak T})$.
The celebrated theorem of Regev \cite{R} gives the exponential growth of the codimension sequence $c_n({\mathfrak R})$, $n=1,2,\ldots$,
for any proper variety of associative algebras. Later Giambruno and Zaicev \cite{GZ1, GZ2} proved that the exponent
\[
\exp({\mathfrak R})=\limsup_{n\to\infty}\sqrt[n]{c_n({\mathfrak R})}
\]
exists and is a nonnegative integer. In \cite{GZ3} they described the minimal varieties $\mathfrak R$ of a given exponent, i.e.,
the varieties $\mathfrak R$ with the property $\exp({\mathfrak S})<\exp({\mathfrak R})$ for any proper subvariety $\mathfrak S$ of $\mathfrak R$.
It has turned out that $\mathfrak G$ and ${\mathfrak T}$ are the only minimal varieties of exponent 2.
But from the point of view of invariant theory
there is a big difference between $\mathfrak G$ and ${\mathfrak T}$. By a result of Domokos and one of the authors \cite{DD1}
the algebra of invariants $F_m^G({\mathfrak R})$ is finitely generated for any reductive group if and only if $\mathfrak R$ satisfies the polynomial identity
of Lie nilpotency $[x_1,\ldots,x_{c+1}]=0$ and this is the case of $\mathfrak G$.
For such varieties the recent paper \cite{DD3} suggests constructive methods which allow to find explicit sets of generators of $F_m^G({\mathfrak R})$.
The variety ${\mathfrak T}$ is the minimal variety of unitary associative algebras
with the property that there exists a reductive group $G$ such that $F_m^G({\mathfrak R})$ is not finitely generated.
Then we consider the more complicated case when $\text{GL}_d({\mathbb C})$ acts on $W_m$ in a noncanonical way and again compute the Hilbert series of
$F_m^G({\mathfrak G})$ and $F_m^G({\mathfrak T})$
for several actions of $\text{GL}_d({\mathbb C})$ and
for $G=\text{Sp}_d({\mathbb C}),\text{O}_d({\mathbb C}),\text{SO}_d({\mathbb C})$.

In a forthcoming paper we calculate the Hilbert series of the algebras of $G$-invariants for different actions of these three classical groups
on several $m$-generated algebras (also for $m\not=d$):
the relatively free algebras of the varieties of associative algebras $\mathfrak M$ and $\text{var}(E\otimes E)$
which are minimal in the class of varieties of exponent 4,
of three varieties of Lie algebras: the metabelian variety ${\mathfrak A}^2$, the center-by-metabelian variety $[{\mathfrak A}^2,{\mathfrak E}]$, and
the variety $\text{var}(\text{sl}_2({\mathbb C}))$ generated by the algebra of $2\times 2$ traceless matrices.
We calculate also the Hilbert series of $E_m^G$, $U^G(F_m({\mathfrak N}_2))$, and $\text{Cl}_m^G(W_p)$ for the same groups $G$.

\section{The method}
For a background and details on representation theory of the general linear group $\text{GL}_d({\mathbb C})$ we refer to the book by Macdonald \cite{Mc}.
Every polynomial $\text{GL}_d({\mathbb C})$-module is a direct sum of irreducible submodules.
The irreducible $\text{GL}_d({\mathbb C})$-modules are $V_d(\lambda)$, where
\[
\lambda=(\lambda_1,\ldots,\lambda_d),\quad \lambda_1\geq\cdots\geq \lambda_d\geq 0,
\]
is a partition in not more than $d$ parts. The Hilbert series
\[
H_{\text{GL}_d}(V_d(\lambda),T_d)=\sum_{\alpha_i\geq 0}\dim V_d^{(\alpha)}(\lambda)t_1^{\alpha_1}\cdots t_d^{\alpha_d}
\]
of $V_d(\lambda)$ describing the ${\mathbb N}_0^d$-grading induced by the action of $\text{GL}_d({\mathbb C})$ is equal to the Schur function
$s_{\lambda}(T_d)=s_{\lambda}(t_1,\ldots,t_d)$.
In particular, if
\[
W=\sum_{n\geq 0}W^{(n)}=\sum_{n\geq 0}\sum_{\lambda}m_{n\lambda}V_d(\lambda),
\]
where $m_{n\lambda}$ is the multiplicity of $V_d(\lambda)$ in the decomposition of $W^{(n)}$ into a sum of irreducible $\text{GL}_d({\mathbb C})$-submodules, then
\[
H_{\text{GL}_d}(W,T_d,z)=\sum_{n\geq 0}\sum_{\lambda}m_{n\lambda}s_{\lambda}(T_d)z^n.
\]
As in \cite{BBDGK} and in the papers cited there it is convenient to introduce two formal power series called the multiplicity series of $W$ which carry
the information for the $\text{GL}_d({\mathbb C})$-structure of $W$:
\[
M(W,T_d,z)=\sum_{n\geq 0}\left(\sum_{\lambda}m_{n\lambda}T_d^{\lambda}\right)z^n
=\sum_{n\geq 0}\left(\sum_{\lambda}m_{n\lambda}t_1^{\lambda_1}\cdots t_d^{\lambda_d}\right)z^n,
\]
\[
M'(W,U_d,z)=\sum_{n\geq 0}\left(\sum_{\lambda}m_{n\lambda}u_1^{\lambda_1-\lambda_2}\cdots u_{d-1}^{\lambda_{d-1}-\lambda_d}u_d^{\lambda_d}\right)z^n,
\]
where the second multiplicity series is obtained from the first one using the change of variables
\[
u_1=t_1,u_2=t_1t_2,\ldots,u_d=t_1\cdots t_d.
\]

The following easy lemma gives the expression of the Hilbert series of $W^G$ for any subgroup $G$ of $\text{GL}_d({\mathbb C})$.

\begin{lemma}\label{Hilbert series for any group}
Let
\[
W=\sum_{n\geq 0}W^{(n)}=\sum_{n\geq 0}\sum_{\lambda}m_{n\lambda}V_d(\lambda),
\]
be a direct sum of polynomial $\text{\rm GL}_d({\mathbb C})$-modules $W^{(n)}$
and let $G$ be an arbitrary subgroup of $\text{\rm GL}_d({\mathbb C})$.
Then the Hilbert series of the $G$-invariants of $W$ is
\[
H(W^G,z)=\sum_{n\geq 0}\left(\sum_{\lambda}m_{n\lambda}\dim V_d(\lambda)^G\right)z^n.
\]
\end{lemma}

\begin{proof}
If $V^{(1)}_d(\lambda^{(1)})\oplus\cdots\oplus V^{(k)}_d(\lambda^{(k)})\subset W$ and
\[
w=w^{(1)}+\cdots+w^{(k)}\in W^G, \quad w^{(i)}\in V^{(i)}_d(\lambda^{(i)}),
\]
then $g(w)=g(w^{(1)})+\cdots+g(w^{(k)})=w^{(1)}+\cdots+w^{(k)}=w$ for all $g\in G$. Since $g(V^{(i)}_d)=V^{(i)}_d$ we obtain that $g(w^{(i)})=w^{(i)}$,
$w^{(i)}\in(V^{(i)}_d)^G$, and
\[
W^G=\sum_{n\geq 0}(W^{(n)})^G=\sum_{n\geq 0}\sum_{\lambda}m_{n\lambda}V_d^G(\lambda)
\]
which implies the formula for $H(W^G,z)$.
\end{proof}

As a consequence we obtain the following method for the computing the Hilbert series of $W^G$ when
$G=\text{Sp}_d({\mathbb C}),\text{O}_d({\mathbb C}),\text{SO}_d({\mathbb C})$.

\begin{theorem}\label{Hilbert series of classical groups}
In the notation of Lemma {\rm \ref{Hilbert series for any group}}
\[
H(W^{\text{\rm Sp}_d({\mathbb C})},z)=\sum_{n\geq 0}\left(\sum_{\mu^2}m_{n\mu^2}\right)z^n,\quad d=2p,
\]
where the summation runs on all partitions $\mu^2=(\mu_1,\mu_1,\mu_2,\mu_2,\ldots,\mu_p,\mu_p)$
with even length of the columns of the corresponding Young diagram $[\mu^2]$;
\[
H(W^{\text{\rm O}_d({\mathbb C})},z)=\sum_{n\geq 0}\left(\sum_{2\nu}m_{n,2\nu}\right)z^n,
\]
where the sum is on all even partitions $2\nu=(2\nu_1,\ldots,2\nu_d)$, i.e., partitions with even parts;
\[
H(W^{\text{\rm SO}_d({\mathbb C})},z)=\sum_{n\geq 0}\left(\sum_{2\nu}m_{n,2\nu}+\sum_{2\nu+1}m_{n,2\nu+1}\right)z^n,
\]
where the sum is on all even partitions $2\nu=(2\nu_1,\ldots,2\nu_d)$ and all odd partitions
$2\nu+1=(2\nu_1+1,\ldots,2\nu_d+1)$.
\end{theorem}

\begin{proof}
By our paper \cite{DH} for $d=2p$ the irreducible $\text{GL}_d({\mathbb C})$-module $V_d(\lambda)$
contains one-dimensional $\text{Sp}_d({\mathbb C})$-invariant subspace if and only if $\lambda$ is a partition
with even length of the columns of the Young diagram $[\lambda]$ and does not contain
$\text{Sp}_d({\mathbb C})$-invariants otherwise. Together with Lemma \ref{Hilbert series for any group}
this gives the formula for the Hilbert series $H(W^{\text{Sp}_d({\mathbb C})},z)$.
The proof of the other two cases is similar since by \cite{DH}
$\dim V_d^{\text{O}_d({\mathbb C})}(\lambda)=1$ when $\lambda$ is an even partition and
$V_d^{\text{O}_d({\mathbb C})}(\lambda)=0$ otherwise. For $\text{SO}_d({\mathbb C})$ we obtain from \cite{DH} that
$\dim V_d^{\text{SO}_d({\mathbb C})}(\lambda)=1$ when $\lambda$ is either an even or an odd partition and
$V_d^{\text{SO}_d({\mathbb C})}(\lambda)=0$ otherwise.
\end{proof}

The following theorem expresses the Hilbert series of the $G$-invariants
in terms of the multiplicity series
for $G=\text{Sp}_d({\mathbb C}),\text{O}_d({\mathbb C}),\text{SO}_d({\mathbb C})$.

\begin{theorem}\label{Hilbert series via multiplicity series}
Let $W=W^{(0)}\oplus W^{(1)}\oplus W^{(2)}\oplus\cdots$ be a direct sum of polynomial $\text{\rm GL}_d({\mathbb C})$-modules
with multiplicity series $M(W,T_d,z)$ and $M'(W,U_d,z)$. Then the Hilbert series of $W^{\text{\rm Sp}_d({\mathbb C})}$ for $d$ even
and of $W^{\text{\rm O}_d({\mathbb C})}$ and $W^{\text{\rm SO}_d({\mathbb C})}$ are equal to
\[
H(W^{\text{\rm Sp}_d({\mathbb C})},z)=M'(W,0,1,0,1,\ldots,0,1,z);
\]
\[
H(W^{\text{\rm O}_d({\mathbb C})},z)=M_d(W,z),
\]
where $M_d(W,z)$ is defined iteratively by
\[
M_1(W,t_2,\ldots,t_d,z)=\frac{1}{2}(M(W,-1,t_2,\ldots,t_d,z)+M(W,1,t_2,\ldots,t_d,z)),
\]
\[
M_2(W,t_3,\ldots,t_d,z)=\frac{1}{2}(M_1(W,-1,t_3,\ldots,t_d,z)+M_1(W,1,t_3,\ldots,t_d,z)),
\]
\[
\cdots\cdots
\]
\[
M_d(W,z)=\frac{1}{2}(M_{d-1}(W,-1,z)+M_{d-1}(1,z));
\]
\[
H(W^{\text{\rm SO}_d({\mathbb C})},z)=M'_d(W,z),
\]
where $M'_d(W,z)$ is defined iteratively by
\[
M'_1(W,u_2,\ldots,u_d,z)=\frac{1}{2}(M'(W,-1,u_2,\ldots,u_d,z)+M'(W,1,u_2,\ldots,u_d,z)),
\]
\[
M'_2(W,u_3,\ldots,u_d,z)=\frac{1}{2}(M'_1(W,-1,u_3,\ldots,u_d,z)+M'_1(W,1,u_3,\ldots,u_d,z)),
\]
\[
\cdots\cdots
\]
\[
M'_{d-1}(W,u_d,z)=\frac{1}{2}(M'_{d-2}(W,-1,u_d,z)+M'_{d-2}(1,u_d,z)),
\]
\[
M'_d(W,z)=M'_{d-1}(W,1,z).
\]
\end{theorem}

\begin{proof}
The arguments of the proof repeat verbatim the arguments of a similar theorem from \cite{DH}
in the case when $W$ is the symmetric algebra $S(W^{(1)})$ of the $\text{GL}_d({\mathbb C})$-module $W^{(1)}$
because in both cases we start with the multiplicity series of $S(W^{(1)})$ and $W$, respectively,
and take a summation on the partitions $\lambda$ which correspond
to those $\text{GL}_d({\mathbb C})$-modules $V_d(\lambda)$ which contribute to the Hilbert series of the $G$-invariants
with nontrivial $G$-invariants for each of the classical groups
$G=\text{Sp}_d({\mathbb C}),\text{O}_d({\mathbb C}),\text{SO}_d({\mathbb C})$.
\end{proof}

\begin{remark}\label{invariants of SL and UT}
The formulas for the Hilbert series of the invariants of
$\text{Sp}_d({\mathbb C}),\text{O}_d({\mathbb C})$, and $\text{SO}_d({\mathbb C})$
in Theorem \ref{Hilbert series via multiplicity series} are in the spirit of similar formulas for the invariants
of the special linear group $\text{SL}_d({\mathbb C})$ and the unitriangular group $\text{UT}_d({\mathbb C})$
given in \cite{BBDGK}:
\[
H(W^{\text{\rm SL}_d({\mathbb C})},z)=M'(W,0,0,\ldots,0,1,z),
\]
\[
H(W^{\text{\rm UT}_d({\mathbb C})},z)=M(W,1,\ldots,1,z)=M'(W,1,\ldots,1,z).
\]
\end{remark}

\section{Canonical action of $\text{GL}_d({\mathbb C})$}
In this section we compute the Hilbert series of the algebras of invariants when the group $\text{GL}_d({\mathbb C})$
acts canonically on the vector space $W_d={\mathbb C}X_d$ generating the algebra. The algebras in consideration are the relatively free algebras
of the varieties of associative algebras $\mathfrak G$ and $\mathfrak T$. The necessary background including the application of representation theory of
the general linear group to relatively free algebras may be found in the book by one of the authors \cite{D}.
In what follows we assume that the relatively free algebras are freely generated by the set $X_d$.
We also assume that with respect to the basis $X_d=X_{2p}$ the group $\text{Sp}_{2p}({\mathbb C})$ consists of the $(2p)\times (2p)$ matrices
$g$ with the property $g^t\Omega g=\Omega$, where $g^t$ is the transpose of $g$,
$\displaystyle \Omega=\left(\begin{matrix}0&I_p\\
=I_p&0\\
\end{matrix}\right)$, and $I_p$ is the identity $p\times p$ matrix. Similarly, the group $\text{O}_d({\mathbb C})$ consists of the $d\times d$ matrices $g$
such that $g^tg=I_d$ and $\text{SO}_d({\mathbb C})=\{g\in \text{O}_d({\mathbb C})\mid \det(g)=1\}$.

\subsection{The relatively free algebra $F_d({\mathfrak G})$}
The description of the polynomial identities and the cocharacter sequence of the variety $\mathfrak G$
generated by the Grassmann algebra $E$ was given by Krakowski and Regev \cite{KR}
and by Olsson and Regev \cite{OR}.
The variety is defined by the polynomial identity
\[
[x_1,x_2,x_3]=[[x_1,x_2],x_3]=0.
\]
It is well known that $F_d({\mathfrak G})$ has a basis consisting of all
\[
x_1^{n_1}\cdots x_d^{n_d}[x_{i_1},x_{i_2}]\cdots [x_{i_{2q-1}},x_{i_{2q}}],
\]
$n_j\geq 0$, $1\leq i_1<i_2<\cdots<i_{2q-1}<i_{2q}\leq d$, $0\leq 2q\leq d$, and in $F_d({\mathfrak G})$
\[
[x_{\sigma(1)},x_{\sigma(2)}]\cdots [x_{\sigma(2q-1)},x_{\sigma(2q)}]=\text{sign}(\sigma)[x_1,x_2]\cdots[x_{2q-1},x_{2q}],\quad \sigma\in S_{2q}.
\]
The cocharacter sequence of $\mathfrak G$ is
\[
\chi_0({\mathfrak G})=\chi(0),\chi_n({\mathfrak G})=\sum_{i=1}^n\chi(i,1^{n-i}),n=1,2,\ldots,
\]
where $\chi(\lambda)$, $\lambda\vdash n$, is the irreducible $S_n$-character indexed by the partition $\lambda$.
In other words, the summation is on all partitions $\lambda$ with Young diagram consisting of one long row and one long column.
This implies that
\[
F_d({\mathfrak G})={\mathbb C}+\sum_{n\geq 1}\sum_{i=1}^n V_d(i,1^{n-i}),
\]
where $V_d(\lambda)=0$ if $\lambda$ is a partition in more than $d$ parts.

\begin{proposition}\label{F(var E) action of Sp}
Let $d=2p$. Then the Hilbert series of $F_d({\mathfrak G})^{\text{\rm Sp}_d({\mathbb C})}$ is
\[
H(F_d({\mathfrak G})^{\text{\rm Sp}_d({\mathbb C})},z)=1+z^2+z^4+\cdots+z^{2p}
\]
and the algebra $F_d({\mathfrak G})^{\text{\rm Sp}_d({\mathbb C})}$ is generated by
\[
f=[x_1,x_{p+1}]+[x_2,x_{p+2}]+\cdots+[x_p,x_{2p}].
\]
\end{proposition}

\begin{proof}
The first part of the proposition follows immediately from Theorem \ref{Hilbert series of classical groups}
because the only partitions $(i,1^{n-i})$ in not more than $2p$ parts and with even length of the corresponding Young diagram
are $(0),(1^2),(1^4),\ldots,(1^{2p})$. Easy computations show that $f\in F_d({\mathfrak G})^{\text{Sp}_d({\mathbb C})}$ and $f,f^2,\ldots,f^p\not=0$
which immediately implies that $\{1,f,f^2,\ldots,f^p\}$ is a basis of the vector space $F_d({\mathfrak G})^{\text{Sp}_d({\mathbb C})}$.
\end{proof}

\begin{proposition}\label{F(var E) action of O}
The Hilbert series of the algebra $F_d({\mathfrak G})^{\text{\rm O}_d({\mathbb C})}$ is
\[
H(F_d({\mathfrak G})^{\text{\rm O}_d({\mathbb C})},z)=\frac{1}{1-z^2}
\]
and the algebra $F_d({\mathfrak G})^{\text{\rm O}_d({\mathbb C})}$ is the symmetric algebra generated by the element of $F_d({\mathfrak G})$
\[
f=x_1^2+\cdots+x_d^2.
\]
\end{proposition}

\begin{proof}
We repeat the arguments in the proof of Proposition \ref{F(var E) action of Sp}.
Since the only even partitions $(i,1^{n-i})$ are (2q), $q=0,1,2,\ldots$,
applying Theorem \ref{Hilbert series of classical groups} we derive that
\[
H(F_d({\mathfrak G})^{\text{O}_d({\mathbb C})},z)=1+z^2+z^4+\cdots=\frac{1}{1-z^2}.
\]
Since the element $f$ is an $\text{O}_d({\mathbb C})$-invariant and its powers are nonzero in
$F_d({\mathfrak G})$ we conclude that $\{1,f,f^2,\ldots\}$ is a basis of $F_d({\mathfrak G})^{\text{O}_d({\mathbb C})}$
which completes the proof.
\end{proof}

\begin{proposition}\label{F(var E) action of SO}
The Hilbert series of the algebra $F_d({\mathfrak G})^{\text{\rm SO}_d({\mathbb C})}$ is
\[
H(F_d({\mathfrak G})^{\text{\rm SO}_d({\mathbb C})},z)=\frac{1+z^d}{1-z^2}
\]
and the algebra $F_d({\mathfrak G})^{\text{\rm SO}_d({\mathbb C})}$ is generated by the element
\[
f=x_1^2+\cdots+x_d^2
\]
and the standard polynomial of degree $d$
\[
St_d(x_1,\ldots,x_d)=\sum_{\sigma\in S_d}\text{\rm sign}(\sigma)x_{\sigma(1)}\cdots x_{\sigma(d)}.
\]
\end{proposition}

\begin{proof}
As in the proof of Proposition \ref{F(var E) action of O},
Theorem \ref{Hilbert series of classical groups} gives that
the one-dimensional contributions to the algebra $F_d({\mathfrak G})^{\text{O}_d({\mathbb C})}$
come from the even partitions $(i,1^{n-i})=(2q)$, $q=0,1,2,\ldots$,
and from the odd partitions $(i,1^{n-i})=(2k+1,1^{d-1})$, $k=0,1,2,\ldots$,
i.e.,
\[
H(F_d({\mathfrak G})^{\text{SO}_d({\mathbb C})},z)=(1+z^d)(1+z^2+z^4+\cdots)=\frac{1+z^d}{1-z^2}.
\]
Since the standard polynomial $St_d=St_d(x_1,\ldots,x_d)$ is an $\text{SO}_d({\mathbb C})$-invariant we derive that
$F_d({\mathfrak G})^{\text{SO}_d({\mathbb C})}$ has a basis
\[
\{1,f,f^2,\ldots\}\cup\{St_d,St_df,St_df^2,\ldots\}
\]
and hence is generated by $f$ and $St_d$.
\end{proof}

\begin{remark}\label{F(var E) action of SL and UT}
Applying ideas from \cite{BBDGK} we obtain that
\[
H(F_d({\mathfrak G})^{\text{SL}_d({\mathbb C})},z)=1+z^d
\]
and $F_d({\mathfrak G})^{\text{SL}_d({\mathbb C})}$ has a basis consisting of 1 and $St_d(x_1,\ldots,x_d)$.

For the unitriangular group $\text{UT}_d({\mathbb C})$ we have
\[
H(F_d({\mathfrak G})^{\text{UT}_d({\mathbb C})},z)=(1+z+z^2+\cdots)(1+z^2+z^3+\cdots+z^d)=\frac{1+z^2+z^3+\cdots+z^d}{1-z}
\]
and $F_d({\mathfrak G})^{\text{UT}_d({\mathbb C})}$ is generated by $x_1$ and $St_n(x_1,\ldots,x_n)$, $n=2,3,\ldots,d$.
\end{remark}

\subsection{The relatively free algebra $F_d({\mathfrak T})$}
By a theorem of Maltsev \cite{Ma} the polynomial identities of the algebra of $c\times c$ upper triangular matrices follow from the polynomial identity
\[
[x_1,x_2]\cdots[x_{2c-1},x_{2c}]=0.
\]
In the special case $c=2$ the cocharacter sequence
of the variety ${\mathfrak T}$ was computed by Mishhenko, Regev, and Zaicev \cite{MiRZ}:
\[
\chi_n({\mathfrak T})=\sum_{\lambda\vdash n}m_{\lambda}({\mathfrak T})\chi_{\lambda},
\]
where
\[
m_{\lambda}({\mathfrak T})=\begin{cases}
1, & \text{ if } \lambda = (n);\\
\lambda_1 - \lambda_2 + 1, & \text{ if } \lambda = (\lambda_1, \lambda_2), \lambda_2 > 0; \\
\lambda_1 - \lambda_2 + 1, & \text{ if } \lambda = (\lambda_1, \lambda_2, 1);\\
0 & \text{ in all other cases.}
\end{cases}
\]

\begin{proposition}\label{F(N2A) action of Sp}
Let $d=2p$. Then the Hilbert series of $F_d({\mathfrak T})^{\text{\rm Sp}_d({\mathbb C})}$ is
\[
H(F_d({\mathfrak T})^{\text{\rm Sp}_d({\mathbb C})},z)=\frac{1}{1-z^2}.
\]
The algebra $F_d({\mathfrak T})^{\text{\rm Sp}_d({\mathbb C})}$ is not finitely generated.
A set of generators can be defined inductively by
\[
f_1=[x_1,x_{p+1}]+[x_2,x_{p+2}]+\cdots+[x_p,x_{2p}]=\sum_{i=1}^p[x_i,x_{p+i}],
\]
\[
f_{n+1}=\sum_{i=1}^px_if_nx_{p+i}-x_{p+i}f_nx_i,\quad n=1,2,\ldots .
\]
\end{proposition}

\begin{proof}
As in the previous subsection the nonzero coefficients of the Hilbert series come from the partitions
$\mu^2=(\mu_1,\mu_1,\mu_2,\mu_2,\ldots,\mu_p,\mu_p)$.
In our case these partitions are $\mu^2=(q,q)$, $q=0,1,2,\ldots$, and all they are of multiplicity 1.
This gives the Hilbert series $H(F_d({\mathfrak T})^{\text{\rm Sp}_d({\mathbb C})},z)$.
As in the case of $F_d({\mathfrak G})^{\text{\rm Sp}_d({\mathbb C})}$ it is easy to see that the elements $f_n$,
$n=0,1,2,\ldots$, are $\text{\rm Sp}_d({\mathbb C})$-invariants and they form a basis of
$F_d({\mathfrak G})^{\text{\rm Sp}_d({\mathbb C})}$. Since $f_mf_n=0$ for $m,n>0$, we derive that the algebra
of invariants is not finitely generated.
\end{proof}

\begin{proposition}\label{F(N2A) action of O}
The Hilbert series of the algebra $F_d({\mathfrak T})^{\text{\rm O}_d({\mathbb C})}$ is
\[
H(F_d({\mathfrak T})^{\text{\rm O}_d({\mathbb C})},z)=\frac{1-2z^2+2z^4}{(1-z^2)^3}.
\]
The algebra $F_d({\mathfrak T})^{\text{\rm O}_d({\mathbb C})}$ is not finitely generated.
\end{proposition}

\begin{proof}
Again, the $n$th coefficient of the Hilbert series $H(F_d({\mathfrak T}_2{\mathfrak A})^{\text{\rm O}_d({\mathbb C})},z)$
is equal to the sum of the multiplicities $m_{\lambda}({\mathfrak T})$ of the even partitions $\lambda$ of $n$.
Hence
\[
H(F_d({\mathfrak T})^{\text{O}_d({\mathbb C})},z)=(1+z^2+z^4+\cdots)+z^4(1+z^4+z^8+\cdots)\sum_{i\geq 0}(2i+1)z^{2i}
\]
\[
=\frac{1}{1-z^2}+\frac{z^4}{1-z^4}\frac{d}{dz}\sum_{i\geq 0}z^{2i+1}
=\frac{1}{1-z^2}+\frac{z^4}{1-z^4}\frac{d}{dz}\frac{z}{1-z^2}=\frac{1-2z^2+2z^4}{(1-z^2)^3}.
\]
As in the case of $F_d({\mathfrak G})^{\text{\rm O}_d({\mathbb C})}$ the element
\[
f=x_1^2+\cdots+x_d^2
\]
is an $\text{\rm O}_d({\mathbb C})$-invariant
and its powers give the contribution $1+z^2+z^4+\cdots$ to the Hilbert series. The $\text{\rm O}_d({\mathbb C})$-invariants
in the commutator ideal $F_d'({\mathfrak T})$ of $F_d({\mathfrak T})$ form an
$S({\mathbb C}f)$-bimodule. If this bimodule is generated by the homogeneous system $\{w_j\mid j\in J\}$, then
$F_d'({\mathfrak T})^{\text{\rm O}_d({\mathbb C})}$ is spanned as a vector space by
\[
\{f^qw_jf^r\mid q,r\geq 0,j\in J\}
\]
and the coefficients of the Hilbert series are bounded from above by the coefficients of the series
\[
\frac{1}{1-z^2}+\frac{1}{(1-z^2)^2}\sum_{j\in J}z^{\deg(w_j)}.
\]
Comparing this expression with the already computed Hilbert series we obtain
\[
\sum_{j\in J}z^{\deg(w_j)}\geq \frac{z^4}{1-z^4}=z^4+z^8+z^{12}+\cdots,
\]
where the inequality between the series means an inequality between the corresponding coefficients.
Since $F_d'({\mathfrak T})^2=0$ this implies that the algebra
$F_d({\mathfrak T})^{\text{\rm O}_d({\mathbb C})}$ is not finitely generated.
\end{proof}

The proof of the following proposition is similar to the proof of the previous one.

\begin{proposition}\label{F(N2A) action of SO}
The Hilbert series of the algebra $F_d({\mathfrak T})^{\text{\rm SO}_d({\mathbb C})}$ is
\[
H(F_d({\mathfrak T})^{\text{\rm SO}_d({\mathbb C})},z)=
\begin{cases}\displaystyle \frac{1-z^2+2z^4}{(1-z^2)^3},&\text{ if }d=2;\\
\displaystyle \frac{1-2z^2+z^3+2z^4}{(1-z^2)^3},&\text{ if }d=3;\\
H(F_d({\mathfrak T})^{\text{\rm O}_d({\mathbb C})},z),&\text{ if }d>3.
\end{cases}
\]
The algebra $F_d({\mathfrak T})^{\text{\rm SO}_d({\mathbb C})}$ is not finitely generated.
\end{proposition}

\begin{remark}\label{F(N2A) action of SL and UT}
As in Remark \ref{F(var E) action of SL and UT} one can compute the Hilbert series of
$F_d({\mathfrak T})^{\text{\rm SL}_d({\mathbb C})}$ and
$F_d({\mathfrak T})^{\text{\rm UT}_d({\mathbb C})}$:
\[
H(F_d({\mathfrak T})^{\text{\rm SL}_d({\mathbb C})},z)=
\begin{cases}
\displaystyle \frac{1}{1-z^2},&\text{ if }d=2;\\
1+z^3,&\text{ if }d=3;\\
1,&\text{ if }d>3,
\end{cases}
\]
\[
H(F_d({\mathfrak T})^{\text{\rm UT}_d({\mathbb C})},z)=
\begin{cases}
\displaystyle \frac{1}{1-z}+\frac{z^2}{(1-z)^2(1-z^2)},&\text{ if }d=2;\\
\displaystyle \frac{1}{1-z}+\frac{z^2}{(1-z)^3},&\text{ if }d\geq 3.\\
\end{cases}
\]
The algebras $F_2({\mathfrak T})^{\text{\rm SL}_2({\mathbb C})}$
and $F_d({\mathfrak T})^{\text{\rm UT}_d({\mathbb C})}$ are not finitely generated.
\end{remark}

\section{Other actions of $\text{\rm GL}_d({\mathbb C})$}
In this section we compute the Hilbert series of the algebras $F_m({\mathfrak G})^G$ and $F_m({\mathfrak T})^G$
when $G=\text{\rm Sp}_d({\mathbb C})$ ($d$ even), $\text{\rm O}_d({\mathbb C})$, and $\text{\rm SO}_d({\mathbb C})$
and for several noncanonical actions of the group $\text{\rm GL}_d({\mathbb C})$ on $W_m$. The most important step of
the calculations is to find the multiplicity series $M(F_m({\mathfrak G}),T_d,z)$ and $M(F_m({\mathfrak T}),T_d,z)$
and their counterparts $M'(F_m({\mathfrak G}),U_d,z)$ and $M'(F_m({\mathfrak T}),U_d,z)$). These computations use the methods in \cite{BBDGK}.

\subsection{The algebra $F_m({\mathfrak G})$}
The Hilbert series of the algebra $F_m({\mathfrak G})$ which counts the canonical action of $\text{GL}_m({\mathbb C})$ is
\[
H_{\text{GL}_m({\mathbb C})}(F_m({\mathfrak G}),T_m,z)=\frac{1}{2}+\frac{1}{2}\prod_{i=1}^m\frac{1+t_iz}{1-t_iz}.
\]

\begin{example}\label{F(E) lambda (2)}
Let as a $\text{GL}_2({\mathbb C})$-module $W_3$ be isomorphic to $V_2(2)$. Then
\[
H_{\text{GL}_2({\mathbb C})}(F_3({\mathfrak G}),T_2,z)=\frac{1}{2}+\frac{(1+t_1^2z)(1+t_1t_2z)(1+t_2^2z)}{2(1-t_1^2z)(1-t_1t_2z)(1-t_2^2z)},
\]
\[
M(F_3({\mathfrak G}),T_2,z)=\frac{1}{1-t_1^2z}+\frac{t_1^2t_2(t_1+t_2)z^2}{(1-t_1^2z)(1-t_1t_2z)},
\]
\[
M'(F_3({\mathfrak G}),U_2,z)=\frac{1}{1-u_1^2z}+\frac{u_2(u_1^2+u_2)z^2}{(1-u_1^2z)(1-u_2z)}.
\]
Applying Theorem \ref{Hilbert series via multiplicity series} we obtain
\[
H(F_3({\mathfrak G})^{\text{\rm Sp}_2({\mathbb C})},z)=M'(F_3({\mathfrak G}),0,1,z)
=1+\frac{z^2}{1-z},
\]
\[
M_1(F_3({\mathfrak G}),t_2,z)=\frac{1}{2}(M(F_3({\mathfrak G}),-1,t_2,z)+M(F_3({\mathfrak G}),1,t_2,z))
\]
\[
=\frac{1}{1-z}+\frac{t_2^2z^2(1+z)}{(1-z)(1-t_2^2z^2)},
\]
\[
H(F_3({\mathfrak G})^{\text{\rm O}_d({\mathbb C})},z)=M_2(F_3({\mathfrak G}),z)
=\frac{1}{1-z}+\frac{z^2}{(1-z)^2},
\]
\[
M'_1(F_3({\mathfrak G}),u_2,z)=\frac{1}{2}(M'(F_3({\mathfrak G}),-1,u_2,z)+M'(F_3({\mathfrak G}),1,u_2,z))
\]
\[
=\frac{1}{1-z}+\frac{u_2(1+u_2)z^2}{(1-z)(1-u_2z)},
\]
\[
H(F_3({\mathfrak G})^{\text{\rm SO}_2({\mathbb C})},z)=M'_2(F_3({\mathfrak G}),z)=M'_1(F_3({\mathfrak G}),1,z)
=\frac{1}{1-z}+\frac{2z^2}{(1-z)^2}.
\]
\end{example}

\begin{example}\label{F(E) lambda (1+1)}
Let as a $\text{GL}_2({\mathbb C})$-module $W_4$ be isomorphic to $V_2(1)\oplus V_2(1)$. Then
\[
H_{\text{GL}_2({\mathbb C})}(F_4({\mathfrak G}),T_2,z)=\frac{1}{2}+\frac{(1+t_1z)^2(1+t_2z)^2}{2(1-t_1z)^2(1-t_2z)^2},
\]
\[
M(F_4({\mathfrak G}),T_2,z)=\frac{1+t_1(t_1+3t_2)z^2+2t_1^2t_2z^3+t_1^2t_2(-t_1+4t_2)z^4-2t_1^3t_2^2z^5}{(1-t_1z)^2(1-t_1t_2z^2)},
\]
\[
M'(F_3({\mathfrak G}),U_2,z)=\frac{1+(u_1^2+3u_2)z^2+2u_1u_2z^3+u_2(-u_1^2+4u_2)z^4-2u_1u_2^2z^5}{(1-u_1z)^2(1-u_2z^2)}.
\]
\[
H(F_4({\mathfrak G})^{\text{\rm Sp}_2({\mathbb C})},z)=M'(F_4({\mathfrak G}),0,1,z)
=\frac{1+3z^2+4Z^4}{1-z^2},
\]
\[
H(F_4({\mathfrak G})^{\text{\rm O}_2({\mathbb C})},z)=\frac{1+z^2+7z^4-z^6}{(1-z^2)^3},
\]
\[
H(F_4({\mathfrak G})^{\text{\rm SO}_2({\mathbb C})},z)=\frac{1+5z^2+11z^4-z^6}{(1-z^2)^3}.
\]
\end{example}

\subsection{The algebra $F_m({\mathfrak T})$}
The Hilbert series of $F_m({\mathfrak T})$ is
\[
H_{\text{GL}_m({\mathbb C})}(F_m({\mathfrak T}),T_m,z)=2\prod_{i=1}^m\frac{1}{1-t_iz}+((t_1+\cdots+t_m)z-1)\prod_{i=1}^m\frac{1}{(1-t_iz)^2}.
\]
Most of the multiplicity series for $F_m({\mathfrak T})$ in the cases considered in the sequel
were computed in \cite{BBDGK} using the Hilbert series $H_{\text{\rm GL}_m({\mathbb C})}(F_m({\mathfrak T}),T_m,z)$.

\begin{example}\label{F3(T) lambda (1,1)}
Let as a $\text{GL}_3({\mathbb C})$-module $W_3$ be isomorphic to $V_3(1^2)$. Then
\[
M(F_3({\mathfrak T}),T_3,z)=\frac{1-t_1t_2z+t_1^2t_2^2(t_1+t_3)t_3z^3}{(1-t_1t_2z)^2(1-t_1^2t_2t_3z^2)},
\]
\[
M'(F_3({\mathfrak T}),U_3,z)=\frac{1-u_2z+(u_1u_2+u_3)u_3z^3}{(1-u_2z)^2(1-u_1u_3z^2)},
\]
\[
H(F_3({\mathfrak T})^{\text{O}_3({\mathbb C})},z)=H(F_3({\mathfrak T})^{\text{SO}_3({\mathbb C})},z)=\frac{1-z-z^2+2z^3}{(1-z)(1-z^2)^2}.
\]
\end{example}

\begin{example}\label{F4(T) lambda (1+1)}
Let as a $\text{GL}_2({\mathbb C})$-module $W_4$ be isomorphic to $V_2(1)\oplus V_2(1)$. Then
\[
M(F_4({\mathfrak T}),T_2,z)=\frac{1}{(1-t_1z)^4(1-t_1t_2z^2)^5}(1-2t_1z+t_1(2t_1-t_2)z^2+10t_1^2t_2z^3
\]
\[
+3t_1^2t_2(-3t_1+4t_2)z^4-34t_1^3t_2^2z^5
+t_1^3t_2^2(19t_1-8t_2)z^6+18t_1^4t_2^3z^7
\]
\[
+2t_1^4t_2^3(-4t_1+t_2)z^8-4t_1^5t_2^4z^9+2t_1^6t_2^4z^{10}),
\]
\[
M'(F_4({\mathfrak T}),U_2,z)=\frac{1}{(1-u_1z)^4(1-u_2z^2)^5}(1-2u_1z-u_2z^2+10u_1u_2z^3
\]
\[
+3u_2(-3u_1^2+4u_2)z^4-34u_1u_2^2z^5+u_2^2(19u_1^2-8u_2)z^6+2u_1^2z^2+18u_1u_2^3z^7
\]
\[
+2u_2^3(u_2-4u_1^2)z^8-4u_1u_2^4z^9+2u_1^2u_2^4z^{10}),
\]
\[
H(F_4({\mathfrak T})^{\text{Sp}_2({\mathbb C})},z)= \frac{1-z^2+12z^4-8z^6+2z^8}{(1-z^2)^5},
\]
\[
H(F_4({\mathfrak T})^{\text{O}_2({\mathbb C})},z)=\frac{1-3z^2+30z^4+2z^8}{(1-z^2)^7},
\]
\[
H(F_4({\mathfrak T})^{\text{SO}_2({\mathbb C})},z)=\frac{1+z^2+43z^4+19z^6-6z^8+2z^{10}}{(1-z^2)^7}.
\]
\end{example}

\begin{example}\label{F3(T) lambda (2)}
Let as a $\text{GL}_2({\mathbb C})$-module $W_3$ be isomorphic to $V_2(2)$. Then
\[
M(F_3({\mathfrak T}),T_2,z)=\frac{1-t_1(t_1+t_2)z+t_1^2t_2(2t_1-t_2)z^2+2t_1^3t_2^2(t_1+t_2)z^3-2t_1^5t_2^3z^4}{(1-t_1^2z)^2(1-t_1t_2z)(1-t_1^2t_2^2z^2)^2},
\]
\[
M'(F_3({\mathfrak T}),U_2,z)=\frac{1-(u_1^2+u_2)z+(2u_1^2-u_2)u_2z^2+2(u_1^2+u_2)u_2^2z^3-2u_1^2u_2^3z^4}{(1-u_1^2z)^2(1-u_2z)(1-u_2^2z^2)^2},
\]
\[
H(F_3({\mathfrak T})^{\text{O}_3({\mathbb C})},z)=\frac{1-2z+4z^3-2z^4}{(1-z)^3(1-z^2)^2},
\]
\[
=H(F_3({\mathfrak T})^{\text{SO}_3({\mathbb C})},z)=\frac{1-2z+z^2+4z^3-2z^4}{(1-z)^3(1-z^2)^2}.
\]
\end{example}

\begin{example}\label{F6(T) lambda (2)}
Let as a $\text{GL}_3({\mathbb C})$-module $W_6$ be isomorphic to $V_3(2)$. Then
\[
H(F_6({\mathfrak T})^{\text{O}_3({\mathbb C})},z)=H(F_6({\mathfrak T})^{\text{SO}_3({\mathbb C})},z)
\]
\[
=\frac{1-2z-z^2+4z^3+6z^4+2z^5-12z^6+z^7+6z^8+4z^9-2z^{10}-4z^{11}+2z^{12}}{((1-z)(1-z^2)(1-z^3))^3}.
\]
\end{example}

\section*{Acknowledgements}
The authors are very grateful to M\'aty\'as Domokos for the stimulating discussions.

\end{document}